\documentclass[11pt]{scrartcl}
\usepackage[T1]{fontenc}
\usepackage[utf8]{inputenc}

\usepackage[english]{babel}
\usepackage{lmodern}
\usepackage{hyperref}
\hypersetup{colorlinks = true, citecolor = blue}
\usepackage{amsmath, amssymb, amsthm}
\usepackage{mathrsfs, stmaryrd, mathtools}
\usepackage{cleveref}

\usepackage{tikz}
\usepackage{pgfplots}
\pgfplotsset{compat=1.16}
\usepgfplotslibrary{groupplots}
\DeclareUnicodeCharacter{00A0}{ }

\theoremstyle{plain}
\newtheorem{theorem}{Theorem}
\newtheorem{lemma}[theorem]{Lemma}
\newtheorem{proposition}[theorem]{Proposition}
\newtheorem{fact}[theorem]{Fact}
\newtheorem*{fact*}{Fact}
\theoremstyle{remark}
\newtheorem{remark}{Remark}

\newcommand\ceff{\mathscr{C}}
\newcommand\ceffprt{\beta}
\renewcommand\root{\text{\o}}
\newcommand\cd[1][t]{\mathsf{c}^{#1}}
\newcommand\res[1][t]{\mathsf{r}^{#1}}
\newcommand\numch[2][t]{\nu_{#2}^{#1}}
\newcommand\wt[1][t]{\mathsf{A}^{#1}} 
\newcommand\prt[1]{{#1}_{*}}
\newcommand\N{\mathbb{N}}

\newcommand\R{\mathbb{R}}

\newcommand\words{\mathcal{U}}
\newcommand\dd{\mathop{}\!\mathrm{d}}

\DeclarePairedDelimiter{\renorm}{\langle}{\rangle}
\DeclarePairedDelimiter{\abs}{\lvert}{\rvert}
\DeclarePairedDelimiterX\intff[2]{[}{]}{#1,#2}
\DeclarePairedDelimiterX\intfo[2]{[}{)}{#1,#2}
\DeclarePairedDelimiterX\intof[2]{(}{]}{#1,#2}
\DeclarePairedDelimiterX\intoo[2]{(}{)}{#1,#2}
\DeclarePairedDelimiter{\pars}{(}{)}
\DeclarePairedDelimiter{\bracks}{[}{]}
\DeclarePairedDelimiter{\braces}{\lbrace}{\rbrace}
\DeclarePairedDelimiterX{\setof}[2]{\lbrace}{\rbrace}{#1\,{:}\,#2}
\DeclarePairedDelimiterX{\bracksof}[2]{[}{]}{#1\,\delimsize\vert\,#2}
\DeclarePairedDelimiterX{\parsof}[2]{(}{)}{#1\,\delimsize\vert\,#2}
\DeclarePairedDelimiterXPP\lnorm[2]{}\lVert\rVert{_{#1}}{#2}
\DeclarePairedDelimiter{\floor}{\lfloor}{\rfloor}
\DeclarePairedDelimiter{\ceil}{\lceil}{\rceil}

\newcommand\E{\mathbb{E}}
\newcommand{\Pq}{\mathrm{P}} 
\newcommand{\Eq}{\mathrm{E}} 
\newcommand{\Et}{\mathbf{E}} 
\newcommand{\Pt}{\mathbf{P}} 
\newcommand\GW{\mathbf{GW}}

\newcommand\mi{\mathop{\wedge}} 
\newcommand\oforder{\mathrel{\asymp}}

\newenvironment{acknowledgements}{\par\noindent\textbf{Acknowledgements: }}{}
\newenvironment{classification}{%
  \noindent{\slshape\bfseries 2010 Mathematics Subject Classification.}}{}
\newenvironment{keywords}{\noindent{\slshape\bfseries Keywords.}}{}
\renewenvironment{abstract}{%
  \small
  \begin{center}
    \textbf{Abstract}
  \end{center}
  \par
  }{}

  \date{April 26, 2023}
\title{Conductance of a subdiffusive random weighted tree}
\author{Pierre Rousselin\thanks{LAGA, université Paris 13}}

\begin{document}
\maketitle
\begin{abstract}
  We work on a Galton--Watson tree with random weights, in the so-called
  ``subdiffusive'' regime.
  We study the rate of decay of the conductance between the root and the $n$-th
  level of the tree, as $n$ goes to infinity, by a mostly analytic method.
  It turns out the order of magnitude of the expectation of this conductance can
  be less than $1/n$ (in contrast with the results of Addario-Berry--Broutin--Lugosi
  and Chen--Hu--Lin), depending on the value of the second zero
  of the characteristic function associated to the model.

  We also prove the almost sure (and in $L^p$ for some $p>1$)
  convergence of this conductance divided by its expectation towards the limit
  of the additive martingale.

  \begin{keywords}
    Random walks in random environments, Galton--Watson trees,
    conductance.
  \end{keywords}

  \begin{classification}
    60J80, 60G50, 60F25, 60F15.
  \end{classification}

\end{abstract}

\section{Introduction}
The strong links between electric networks and reversible random walks on graphs
have emerged during the second half of the last century and were popularized in the
seminal book \cite{doyle_snell}.
The special cases of random walks on
(random) trees have been thouroughly studied during the 90's by Lyons
(\cite{Lyons_rwperco,lyons_1992}) and Lyons, Pemantle and Peres
(\cite{LPP95,LPP_biased}), making good use of the electric networks theory.

More recent works on \emph{transient} $\lambda$-biased
random walks on (Galton--Watson) trees show that the effective conductance of the
tree is key to understanding the asymptotic behavior of the walk
(see \cite{Elie_speed} for the speed and \cite{shen_lin_harmonic_biased,Rou2018}
for the dimension of the harmonic measure).

In this paper, we deal with a \emph{null-recurrent}
model of random walk on a Galton-Watson
random weighted tree in a regime called ``subdiffusive'' (see below for
definitions).
The effective conductance of the whole tree is zero by recurrence of the walk
and we are interested in the rate of decay of the conductance between the root
of the tree and the vertices at height $n$, as $n$ goes to infinity.  This gives
the order of magnitude of the probability that the walk hits level $n$ before
returning to the root (see below for details).

\subsection{Conductance of a tree}
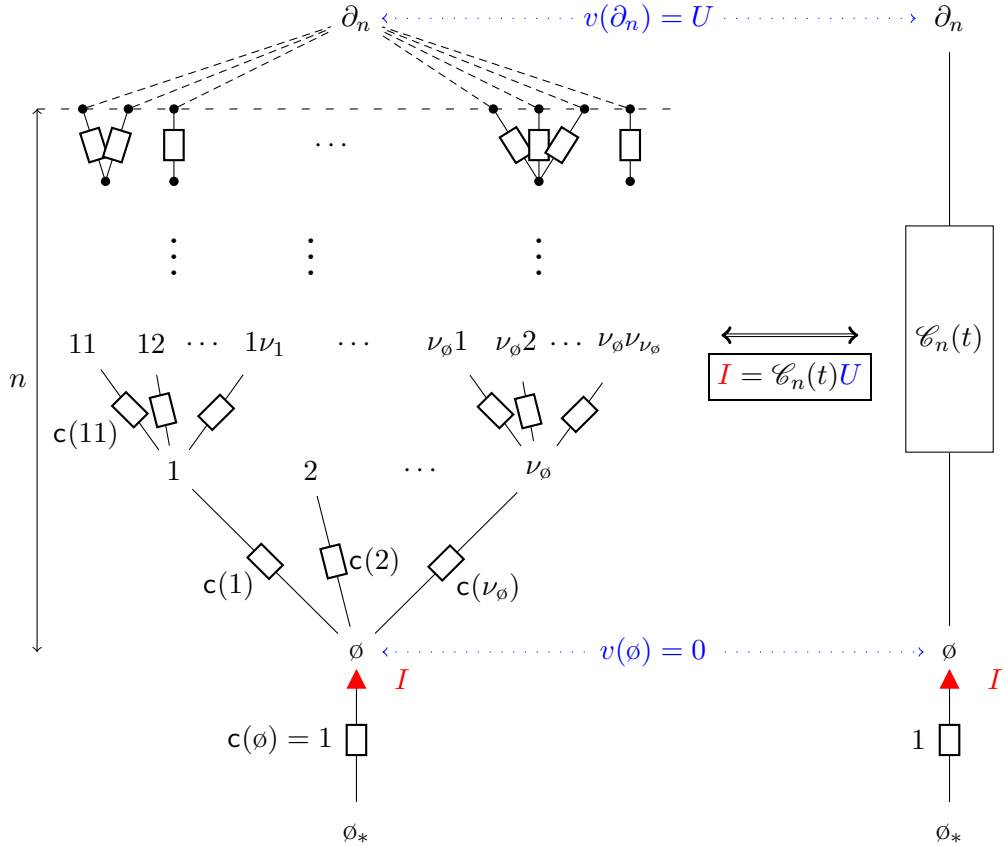
\begin{figure}
  \begin{tikzpicture}[
    scale=1.2,
    resi/.style={
      rectangle,
      minimum width=4mm,
      fill=white,
      draw,
      thick,
      midway,
      sloped
    },
    vert/.style={
      circle,
      minimum size=5mm,
      fill=white
    }, 
    verthaut/.style={
      circle,
      draw,
      fill=black,
      scale=.3
    }
    ]
    \node (root) at (0,0) [vert] {$\root$};
    \node (prtroot) at (0,-2) [vert] {$\prt\root$} ;
    \node (1) at (-2, 2) [vert] {$1$} ;
    \node (2) at (-.5, 2) [vert] {$2$} ;
    \node (der) at (2, 2) [vert] {$\nu_\root$} ;
    \path (2) -- node [midway] {$\cdots$} (der);
    \draw (root) -- node [resi] {} (prtroot);
    \draw (root) -- node [resi] {} (1);
    \draw (root) -- node [resi] {} (2);
    \draw (root) -- node [resi] {} (der);
    \filldraw [color=red] (0,-0.2) -- ++(-0.1,-0.2) -- ++(0.2,0) -- cycle;
    \node [color=red] at (.5, -.3) {$I$};

    \path (prtroot) -- node [midway, left=5] {$\cd[](\root) = 1$} (root);
    \path (root) -- node [midway, below left] {$\cd[](1)$} (1) ;
    \path (root) -- node [midway, below right] {$\cd[](\nu_\root)$} (der);
    \path (root) -- node [midway, right = 2] {$\cd[](2)$} (2);

    \begin{scope}[shift=(1), xscale=.5, yscale=.7]
      \node (11) at (-2, 2) [vert] {$11$} ;
      \node (12) at (-.5, 2) [vert] {$12$} ;
      \node (1der) at (2, 2) [vert] {$1\nu_1$} ;
      \path (12) -- node [midway] {$\cdots$} (1der);
      \draw (1) -- node [resi] {} (11);
      \draw (1) -- node [resi] {} (12);
      \draw (1) -- node [resi] {} (1der);
    \end{scope}
    \begin{scope}[shift=(der), xscale=.5, yscale=.7]
      \node (der1) at (-2, 2) [vert] {$\nu_\root1$} ;
      \node (der2) at (-.5, 2) [vert] {$\nu_\root2$} ;
      \node (derder) at (2, 2) [vert] {$\nu_\root\nu_{\nu_\root}$} ;
      \path (der2) -- node [midway] {$\cdots$} (derder);
      \draw (der) -- node [resi] {} (der1);
      \draw (der) -- node [resi] {} (der2);
      \draw (der) -- node [resi] {} (derder);
    \end{scope}
    \path (1der) -- node [midway] {$\cdots$} (der1);
    \path (1) -- node [midway, below left] {$\cd[](11)$} (11);

    \node[scale=1.5] at (-2, 4.5) {$\vdots$};
    \node[scale=1.5] at (-.5, 4.5) {$\vdots$};
    \node[scale=1.5] at (2, 4.5) {$\vdots$};

    \draw[loosely dashed] (-3.5, 6) -- (3.5, 6);
    \node (A) at (-3,6) [verthaut] {};
    \node (B) at (-2.5,6) [verthaut] {};
    \node (AB) at (-2.75,5.2) [verthaut] {};
    \draw (AB) -- node [resi] {} (A);
    \draw (AB) -- node [resi] {} (B);
    \node (C) at (-2,6) [verthaut] {};
    \node (CC) at (-2, 5.2) [verthaut] {};
    \draw (CC) -- node [resi] {} (C);
    \node (D) at (1.5,6) [verthaut] {};
    \node (E) at (2,6) [verthaut] {};
    \node (F) at (2.5,6) [verthaut] {};
    \node (DEF) at (2, 5.2) [verthaut] {};
    \draw (DEF) -- node [resi] {} (D);
    \draw (DEF) -- node [resi] {} (E);
    \draw (DEF) -- node [resi] {} (F);
    \node (G) at (3,6) [verthaut] {};
    \node (GG) at (3,5.2) [verthaut] {};
    \draw (GG) -- node [resi] {} (G);
    \node at (-.25, 5.6) {$\cdots$};

    \node (ajout) at (0,7) {$\partial_n$};
    \foreach \x in {(A), (B), (C), (D), (E), (F), (G)}
      \draw [densely dashed] \x -- (ajout);

    \draw [<->] (-3.5, 0) -- node [midway, left] {$n$} (-3.5,6);

    \begin{scope}[xshift=6.5cm]
      \node (rt) at (0,0) [vert] {$\root$};
      \node (aj) at (0,7) [vert] {$\partial_n$};
      \node (prtrt) at (0,-2) [vert] {$\prt\root$};
      \draw (rt) -- 
        node [midway, rectangle, minimum height=30mm, draw, fill=white]
        {$\ceff_n(t)$} (aj);
    \node [color=red] at (.5, -.3) {$I$};
    \draw  (prtrt) -- node [resi] {} (rt);
    \filldraw [color=red] (0,-0.2) -- ++(-0.1,-0.2) -- ++(0.2,0) -- cycle;
    \path (prtrt) -- node [midway, left=5] {$1$} (rt);
    \end{scope}

    \draw[loosely dotted, <->, color=blue] (ajout) -- node [midway, color=blue, fill=white] {$v(\partial_n)=U$} (aj);
    \draw[loosely dotted, <->, color=blue](root) -- node [midway, color=blue,
    fill=white]  {$v(\root)=0$} (rt);
    \draw[double, <->] (4,3.5) -- node [midway, below=3] 
    { \framebox{${\color{red}I} = \ceff_n(t) {\color{blue}U}$}} (5.5,3.5);

  \end{tikzpicture}
  \caption{\label{fig:electric_tree}On the left, a rooted tree of height $n$
    with an
    artificial parent of the root and equipped with conductances. The
    potential is $U$ at height $n$ and $0$ at the root. On the right, the
    equivalent reduced electrical network.}
\end{figure}

We first briefly recall some notions of electric networks in the case
where the network is a 
locally finite tree $t$, rooted at some vertex $\root$.
For more detailed and general statements about this theory,
see~\cite{doyle_snell} or \cite{LyonsPeres_book}.
For any vertex $x$ of $t$, associate to the edge between $x$ and its parent
$\prt{x}$ a \emph{conductance} $\cd[](x) \in \intoo{0}{\infty}$, or
alternatively a \emph{resistance} $\res[](x)$ equal to the inverse of the
conductance.
For convenience, we add an artificial parent of the root, denoted by $\prt\root$
and let $\cd[](\root) = 1$ (this is to make the root ``less special'').

Now we fix some positive integer $n$ and assume that the height of $t$ is at
least $n$. We impose a certain fixed electric potential $U$ at the
vertices at height $n$ in $t$, while the potential at the root is $0$.
As another point of view, we may connect the vertices at height $n$ to a new vertex
$\partial_n$, the new edges having infinite conductance (zero resistance), and impose
$v(\partial_n)
= U$.
This defines an electric potential $v$ on the vertices of $t$
between $\root$ and the $n$-th level of $t$.
To be more formal, we need some notations. For a vertex $x$ of $t$,
let us denote by $\prt{x}$ its parent, by $\abs{x}$ its height ($\abs{\root} =
0$), by $\numch[]{x}$ its number of children, by $x1$, $x2$, \dots,
$x{\numch[]{x}}$ its children and by
$\pi(x)$ the sum of the conductances of the edges that are incident to $x$.
The potential $v$ defined on the $n$ first levels of the tree satisfies:
\[
  v(x) =
  \begin{cases}
    0 & \text{if $x = \root$;}\\
    U & \text{if $\abs{x} = n$;}\\
    \frac{1}{\pi(x)}\pars[bigg]{ \cd[](x) v(\prt{x}) 
    + \sum_{j=1}^{\numch[]{x}} \cd[](xj)v(xj)} & \text{if $1 \leq
    \abs{x} \leq n-1$.}
  \end{cases}
\]
The last case in the previous equality
is called \emph{harmonicity} of $v$
at $x$.
Such a potential is well-known to exist and to be unique.
For $x$ in $t$, the electric current $i(x)$ flowing in the edge between $x$
and its parent $\prt{x}$ is defined by Ohm's law as
\[
  i(x) = \cd[](x) (v(x) - v(\prt{x})).
\]
(The harmonicity condition is the same as Kirchhoff's current law.)
Now let
\[
  I = \sum_{j=1}^{\numch[]{\root}} i(j)
\]
be the total current entering the tree.
It is clear that the function $U \mapsto I$ is linear. The constant
ratio $I / U$ is called
the \emph{effective conductance} of $t$ between $\root$ and its $n$-th
level and is denoted by $\ceff_n(t)$. See Figure~\ref{fig:electric_tree} for a
summary of this discussion.

The effective conductance has a pleasant and useful interpretation in terms of random
walks on the vertices of $t$.
We associate to the conductances of the edges a probability kernel $\Pq$ on $t$ in
the following way:
\begin{align*}
  \Pq(x, xi) &= \frac{\cd[](xi)}{\pi(x)} \quad \text{if $1 \leq i \leq
  \numch[]{x}$} \quad \text{and} \quad
  \Pq(x, \prt{x}) = \frac{\cd[](x)}{\pi(x)}.
\end{align*}
For $x$ in
$t$, we write $\Pq_x$ for a probability measure under which the random
sequence $(X_k)_{k\geq0}$ is a random walk starting from $x$ with probability
kernel $\Pq$ and consider the stopping times
\[
  \tau_x = \inf \setof{s\geq0}{X_s = x}, \quad
  \tau_x^+ = \inf \setof{s\geq1}{X_s = x} \quad \text{and} \quad
  \tau^{(n)} = \inf \setof{s\geq0}{\abs{X_s} = n}.
\]
Then, by the Markov property, the function
\[
  v(x) = \Pq_x ( \tau^{(n)} < \tau_\root )
\]
is the electric potential when the vertices at height $n$ have potential $1$ and
the root has potential $0$.
As a consequence, by definition of the current $i$ and, again, the Markov property,
\[
  \ceff_n(t) =
  I = \sum_{j=1}^{\numch[]{\root}} \cd[](j) \Pq_j ( \tau^{(n)} < \tau_\root )
  = \pi(\root) \Pq_\root(\tau^{(n)} < \tau_\root^+).
\]
The conductance $\ceff(t)$ of the whole tree, equal to the limit of the non-increasing
sequence $(\ceff_n(t))_{n\geq0}$, is then positive if and only if the
associated random walk is transient.

A typical choice for the conductances is $\cd[](x) = \lambda^{-\abs{x}}$, for
some fixed $\lambda > 0$.
It corresponds to the $\lambda$-biased random walk on the vertices of $t$,
introduced in \cite{Lyons_rwperco}. In
words, the walker jumps with \emph{weight} $\lambda$ to the parent of its
current position, and with weight $1$ to one of its children.
If $t$ is the tree in which every vertex has $d \geq 2$ children, this random
walk is transient if and only if $\lambda < d$, and the null recurrent case
$\lambda = d$ may be seen as critical.

In \cite{addario-berry_resi}, this infinite $d$-ary tree is considered with
the set of conductances $\cd[](x) = d^{-\abs{x}} X_x$, where the positive random
variables $(X_x)_{x\in t}$ are {i.i.d}, which
corresponds in some way to a (still recurrent) perturbation around this critical regime.
The authors prove that the expectation
$\Et[\ceff_n(t)]$ is of order $1/n$ as $n$ goes to
infinity\footnote{Actually their result is much more precise, but this suffices for the
purpose of this introduction.}.
This result has been recently extended in \cite{chen_hu_lin}
to the case of an infinite,
random \emph{Galton--Watson} trees.

In this work, we investigate the rate of decay of the sequence of random
variables
$(\ceff_n)$ in another
``critical'' setting known as the \emph{subdiffusive} (\cite{hu_shi_subdiffusive})
regime for Galton-Watson trees with random weights.

\subsection{Subdiffusive random weighted trees}

What we call an (edge-)weighted tree is a (rooted, planar) tree $t$
together with a \emph{weight function} $\wt[t] : t \setminus \{ \root \} \to
\intoo{0}{\infty}$. For a vertex $x \neq \root$ in $t$, $\wt[t](x)$ represents
the weight of the edge connecting $x$ to its parent.

We naturally associate to this weighted tree the following probability kernel:
for $x$ in $t$,
\begin{align*}
  \Pq^t(x, xi) =
  \frac{\wt[t](xi)}{1 + \sum_{j=1}^{\numch[]{x}}\wt[t](xj)}
    \quad \text{if $1 \leq i \leq
  \numch[]{x}$} \quad \text{and} \quad
  \Pq^t(x, \prt{x}) =
  \frac{1}{1 + \sum_{j=1}^{\numch[]{x}}\wt[t](xj)},
\end{align*}
that is, if a random walker is at vertex $x$,
it may jump to the $i$-th child of $x$ with weight $\wt[t](xi)$ and to the
parent of $x$ with weight $1$ (if the weights are all constant equal to
$\lambda^{-1}$, we recover the $\lambda$-biased random walk on $t$).
Recall that we also add a vertex $\prt\root$ to serve as an artificial parent to
the root (and the walk is reflected at $\prt\root$).

This random walk
is easily seen to correspond to the conductance $\cd[t]$ defined by
\[
  \forall x \in t,\quad \cd[t](x) = \prod_{\root \prec y \preceq x} \wt[t](y),
\]
where the product above is indexed by the ancestors of $x$ (including $x$) which
are distinct from $\root$.

To define a \emph{Galton-Watson tree with random weights}
consider a random finite sequence of positive
real numbers
\[
  \mathbf{A} = (A(1), \dotsc, A(\nu)),
\]
whose length $\nu \in \{0, 1, 2, \dotsc \}$ may also be random.
Define the free monoid
\[
  \words = \bigsqcup_{k\geq0} \N^k
\]
of all the finite words on the alphabet $\N = \{1, 2, \dotsc\}$, with the
convention that $\N^0$ contains only the empty word $\root$.
Now consider the family $(\mathbf{A}_x)_{x\in\words}$ of {i.i.d.}~random
sequences indexed by $\words$, whose common distribution is the law of
$\mathbf{A}$.
We build a random weighted tree $T$ in the following way: the root $\root$ of
$T$ has $\nu_\root$ children labelled $1$, $2$, ..., $\nu_\root$,
where $\nu_\root$ is the length of the random
sequence
\[
  \mathbf{A}_\root = (A_\root(1), \dotsc, A_\root(\nu_\root))
\]
and, for $1 \leq i \leq \nu_\root$, the weight $\wt[T](i)$
of the edge between $\root$ and
its child $i$ is $A_\root(i)$.
Then, proceed in the same way for the children $1$, $2$, ..., $\nu_\root$ of the
root,
in order to pick their numbers of children $\nu_1$, $\nu_2$, ...,
$\nu_{\nu_\root}$ and the weights of the corresponding edges :
\[
  \wt[T](11) = A_1(1), \dotsc, \wt[T](1\nu_1) = A_1(\nu_1),
  \wt[T](21) = A_2(1), \dotsc, \wt[T](2\nu_2) = A_2(\nu_2),
  \dotsc,
\]
and so on, so that
the weight of the edge between a vertex $xi$ in $T$ and its parent $x$
is $\wt[T](xi) = A_x(i)$.
Notice that $T$, without its weights, is a Galton-Watson tree whose reproduction law is the
distribution of $\nu$.
For this reason we denote by $\GW$ the distribution of $T$, seen as a random
variable in the space of weighted trees.

This very rich family of random walk in a random environment was introduced in
\cite{lyons_pemantle_rwre} and generalized in \cite{faraud_2011}.
The random walk of probability kernel
$\Pq^t$
may be transient or recurrent, for $\GW$-almost every weighted tree $t$,
depending on whether the convex
characteristic function
\[
  \psi(s) = \log \Et \sum^{\nu}_{i=1} A(i)^s,
  \qquad \forall s \in \R,
\]
stays positive on the interval $\intff{0}{1}$.%
\footnote{%
  For a necessary and sufficient
  condition, additional integrability conditions are needed,
  see \cite{faraud_2011}.
}
Since \cite{lyons_pemantle_rwre}, this model has attracted a lot of attention.
For the transient case, see for instance \cite{aidekon_2007} or
\cite{rou_ddroprwre}.
The recurrent case in general is studied in \cite{faraud_hu_shi} or
\cite{andreoletti_debs_number}, among many others.
The recent article
\cite{andreoletti_chen} focuses on the slow null-recurrent regime.

With a slight abuse of terminology (see below),
we call ``subdiffusive'' this model (and by extension the random tree we
work on) when the following hypotheses are satisfied:
\begin{gather*}
  \label{eq:hyp_normalized} \tag{$H_\text{norm}$}
  \psi(1) = \log \Et \sum^{\nu}_{i=1} A(i) = 0 \quad \text{and}\\
  \psi'(1) \coloneqq \Et
  \bracks*{\sum^{\nu}_{i=1} A(i)\log A(i)}   \in
  \intfo{-\infty}{0}.
  \label{eq:hyppsip1} \tag{$H_\text{derivative}$}
\end{gather*}

To state our main result about the conductance in this case, we need to
introduce the additive martingale, also called Mandelbrot's multiplicative
cascade, or Biggins' martingale,
$(M_n(T))_{n\geq0}$, defined by
\[
  M_n(T) = \sum_{\abs{x} = n} \prod_{\root \prec y \preceq x} \wt[T](y)
  = \sum_{\abs{x} = n} \cd[T](x).
\]
It is easily seen to be a martingale with respect to the filtration
$(\mathcal{F}_n)_{n\geq1}$ defined by
\[
  \mathcal{F}_n = \sigma \setof{\mathbf{A}_x}{\abs{x} \leq n - 1}.
\]
By Biggins' theorem (see also \cite{kahane_peyriere, lyons_biggins})
it converges almost surely and in $L^1$ to a
random variable $M_\infty(T)$ which is positive on the event of non-extinction,
provided we also assume the following integrability hypothesis:
\begin{equation}
  \label{eq:hyp_bigg} \tag{$H_{X\log X}$}
  \Et \bracks[\bigg]{ \pars[\Big]{\sum^{\nu}_{i=1} A(i)} \log^+
\pars[\Big]{\sum^{\nu}_{i=1} A(i)}} < \infty.
\end{equation}
The non-degeneracy of $M_\infty(T)$ also allows to prove that, under our
assumptions, the random walk on
$T$ is almost surely \emph{null-recurrent} (we provide a short proof of this
well-known fact in the appendix for completeness).

Now, we could expect results
similar to \cite{addario-berry_resi,chen_hu_lin}. This is not exactly true:
there will be different behaviors depending on the value of
\[
  \kappa = \inf \setof{s > 1}{\psi(s) = 0} \in \intof{1}{\infty}.
\]
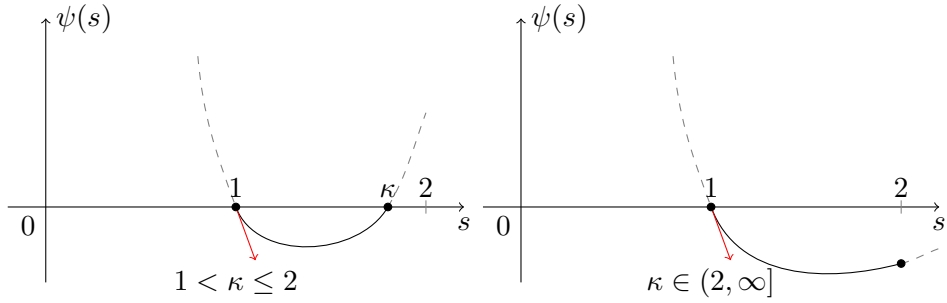
\begin{figure}
  \centering
\begin{tikzpicture}[scale=2.5]
  \draw[->] (-0.2,0) -- (2.2,0);
  \draw[->] (0,-.4) -- (0,1);
  \node[anchor=north east] at (0,0) {$0$};
  \node[anchor=north] at (2.2,0) {$s$};
  \node[anchor=west] at (0,1) {$\psi(s)$};
  \node[anchor=south] at (2,0) {$2$};
  \draw[help lines] (2,-1pt) -- (2,+1pt);
  \draw (1,0) to[out=-70,in=180+60] (1.8,0);
  \draw[->,red] (1,0) -- ++(-70:.3);
  \draw[thin, dashed, gray] (.8,.8) to[out=-85,in=180-70] (1,0);
  \draw[thin, dashed, color=gray] (1.8,0) to[out=60,in=180+70] (2,.5);
  \filldraw (1,0) circle [radius=.02] node[anchor=south] {$1$};
  \filldraw (1.8,0) circle [radius=.02] node[anchor=south] {$\kappa$};
  \node at (1,-.4) {$1 < \kappa \leq 2$};
  \begin{scope}[xshift=2.5cm]
    \draw[->] (-0.2,0) -- (2.2,0);
    \draw[->] (0,-.4) -- (0,1);
    \node[anchor=north east] at (0,0) {$0$};
    \node[anchor=north] at (2.2,0) {$s$};
    \node[anchor=west] at (0,1) {$\psi(s)$};
    \draw (1,0) to[out=-70,in=180+15] (2,-0.3);
    \draw[->,red] (1,0) -- ++(-70:.3);
    \draw[thin, dashed, gray] (.8,.8) to[out=-85,in=180-70] (1,0);
    \draw[thin, dashed, gray] (2,-.3) to[out=15,in=180+20] (2.2,-.22);
    \filldraw (1,0) circle [radius=.02] node[anchor=south] {$1$};
    \filldraw (2,-0.3) circle [radius=.02];
    \draw[help lines] (2,-1pt) -- (2,+1pt);
    \node[anchor=south] at (2,0) {$2$};
    \node at (1,-.4) {$\kappa \in\intof{2}{\infty}$};
  \end{scope}
\end{tikzpicture}
\caption{Schematic behavior of $\psi$ under our hypotheses}
\label{fig:psi_subdiff}
\end{figure}
Our work uses the tail probabilities and some moments of the random variable
$M_\infty(T)$, which
depend on the value of $\kappa$.

For two positive functions $f$ and $g$ defined on a neighborhood of $+\infty$,
we write
$f(x) \oforder g(x)$ as $x$ goes to infinity, when, for some constants $c$ and
$C$ in $\intoo{0}{\infty}$, for $x$ large enough,
\[
  cg(x) \leq f(x) \leq Cg(x).
\]

Under the following integrability hypothesis:
\begin{equation*}
  \label{eq:hypkappa}\tag{$H_\kappa$}
  \begin{gathered}
  \Et \bracks[\Big]{\pars[\Big]{ \sum^{\nu}_{i=1} A(i)}^\kappa}
  +
\Et \bracks[\Big]{ \sum^{\nu}_{i=1} A(i)^\kappa \log^+ A(i)} < \infty,
  \quad
  \text{if $1 < \kappa \leq 2$},
  \\
  \Et \bracks[\Big]{ \pars[\Big]{ \sum^{\nu}_{i=1} A(i)}^2} < \infty, \quad
  \text{if $\kappa \in \intof{2}{\infty}$},
  \end{gathered}
\end{equation*}
we owe to \cite[Theorem~2.1, Theorem~2.2]{liu_generalized} the following
fact:
\begin{fact}\label{fact:liutail}
  If \eqref{eq:hyp_normalized},
  \eqref{eq:hyppsip1} and \eqref{eq:hypkappa} are satisfied, then
  the random variable $M_\infty$ has finite moments of order $p$ for all $p$ in
  $\intfo{1}{\kappa}$ if $\kappa \leq 2$ and for all $p$ in $\intff{1}{2}$ if
  $\kappa > 2$.

  If $\kappa \leq 2$, the asymptotic tail probability of $M_\infty$ satisfies
    \begin{equation}
      \label{eq:tailM}
      \Pt( M_\infty > s ) \oforder_{s\to\infty} s^{-\kappa}.
    \end{equation}
\end{fact}
In the previous statement, as in the rest of this work, we feel free to omit
$T$ as an argument or as a superscript and
write $\ceff_n$ for $\ceff_n(T)$,
$\wt[]$ for $\wt[T]$, $\nu$ for $\nu^T$, \dots, when there is no risk of
confusion.

Throughout this work, we will assume that \eqref{eq:hyp_normalized},
\eqref{eq:hyppsip1} and \eqref{eq:hypkappa} (which supersedes
\eqref{eq:hyp_bigg}) hold.
These assumptions are summed up in Figure~\ref{fig:psi_subdiff}.

One of the most striking result about this regime is
given (under some additional assumptions) in \cite{hu_shi_subdiffusive}:
for $\GW$-almost every infinite $t$, $\Pq^t$-almost surely,
\[
  \lim_{n\to\infty} \frac{\log \max_{0\leq i\leq n}\abs{X_i}}{\log n}
  = 1 - \frac{1}{\kappa \mi 2},
\]
hence the name ``subdiffusive'' in the case $\kappa < 2$, that we improperly
(but conveniently) extend to this whole ``fast, null-recurrent'' case.
A central limit theorem can be found in \cite{faraud_2011}.
More recently, Aïdékon and de Raphélis
(\cite{aidekon_raphelis_scaling,de_raphelis_scaling_subdiffusive})
have proved the joint
convergence of the renormalized height of the walk together with its trace
towards a
continuous-time process and the real forest encoded by this process.

Regarding the conductance, our main result is the following:
\begin{theorem}
  \label{thm}
  Under the hypotheses \eqref{eq:hyp_normalized}, \eqref{eq:hyppsip1} and
  \eqref{eq:hypkappa}, as $n$ goes to infinity,
  \begin{align*}
    \Et[\ceff_n] &\oforder \frac{1}{n^{1/(\kappa - 1)}}
    \quad \text{if $1 < \kappa < 2$ ;}
    \\
    \Et[\ceff_n] &\oforder \frac{1}{n\log n}
    \quad \text{if $\kappa = 2$ and}
    \\
    \Et[\ceff_n] &\sim \frac{1}{n\Et[M_\infty^2]}
= \frac{1 - \Et \bracks[\big]{\sum_{i=1}^{\nu} A(i)^2}}{%
    n \Et \bracks[\Big]{ \sum_{1 \leq i \neq j \leq \nu}^{} A(i)A(j) }}
    \quad \text{if $\kappa > 2$,}
  \end{align*}
  and, in any case, almost surely,
  \[
    \lim_{n\to\infty} \ceff_n / \Et[\ceff_n] = M_\infty.
  \]
  Moreover, the above convergence also holds in $L^p$
  for $p \in \intfo{1}{\kappa}$ if
  $1 < \kappa \leq 2$ and in $L^2$ if $\kappa > 2$.
\end{theorem}

Our method is almost entirely analytic and inspired by
\cite{hu_local_times_subdiff}. The asymptotic analysis of the sequence
$(\Et[\ceff_n])$ relies heavily on \Cref{lem:un}.
\begin{remark}
  The similarity of our almost sure convergence result
  with Theorem 1.1 in \cite{chen_hu_lin} is not very surprising, at least
  heuristically. Indeed, in our setting, for $n \geq 1$, the conductance
  satisfies the distributional equation (see \eqref{eq:ceff2}):
  \begin{equation*}
    \frac{\ceff_n}{\Et[\ceff_n]} = \frac{\Et[\ceff_{n - 1}]}{\Et[\ceff_n]}
    \sum_{i = 1}^\nu A(i) \frac{\ceff_mentioned before Remark~1{n - 1}^{(i)}}{\Et[\ceff_{n - 1}]}\times
    \frac{1}{1 + \ceff_{n - 1}^{(i)}},
  \end{equation*}
  where $\ceff_{n-1}^{(1)}, \ceff_{n-1}^{(2)}, ...$ are i.i.d. copies of $\ceff_n$.
It is thus reasonable to expect an almost sure limit, say $X$, of
$\ceff_n / \Et[\ceff_n]$ to satisfy the distributional recursive equation
\begin{equation*}
  X \stackrel{(d)}{=} \sum_{i = 1}^{\nu} A(i) X^{(i)},
\end{equation*}
where $X^{(1)}, X^{(2)}, ...$ are i.i.d. copies of $X$, independent of
$(A(1), ..., A(\nu))$. In our case, the only solution to this equation is
$M_\infty$.

In~\cite{chen_hu_lin}, the analogous recursive equation (2.3) points towards the
distributional recursive equation satisfied by the limit $W$ of the usual
Galton-Watson martingale.
\end{remark}
\begin{remark}
  Theorem~2.2{} in \cite{liu_generalized} is actually more precise. If one
  further assumes:
  \begin{equation}
    \text{there is no $h > 0$ such that}\quad
    \Pt( (A(1), \dotsc, A(\nu)) \in \N^\nu h) = 1
    \label{eq:hyp_nonlatt}\tag{$H_{\text{non-lattice}}$}
  \end{equation}
  (which is called the ``non-lattice case''), then, there exists $c_\kappa > 0$
  such that
  \[
    \lim_{s\to\infty} s^\kappa \Pt(M_\infty > s) = c_\kappa.
  \]
  Therefore, one could expect that, under the assumption~\eqref{eq:hyp_nonlatt},
  we may also obtain asymptotic equivalences of $\Et[\ceff_n]$,
  for $\kappa \in \intof{1}{2}$.
  Unfortunately, our method was not powerful enough to yield such a
  result: our lower bound of the tail probabilities of
  $\ceff_n/\Et[\ceff_n]$ (see Lemma~\ref{lem:tailceff}) is not sharp enough.
\end{remark}
\begin{acknowledgements}
  This work was part of the author's PhD thesis. The author would like to warmly
  thank
  Yueyun Hu, one of his supervisors, for introducing him to
  this problem and for constant guidance.
  The author also wishes to thank an anonymous referee for his or her very
  thorough work and valuable report.
\end{acknowledgements}

\section{Algebraic identities and lower bound}

In order to use the branching property, we introduce
for any weighted tree $t$,
for any vertex $x$ of $t$, the reindexed subtree of $t$ starting from $x$:
\[
  t[x] = \setof{ y \in \words }{ xy \in t }
  \quad \text{with weights given by} \quad
  \wt[{t[x]}](y) = \wt[t](xy), \quad \forall y\in t[x].
\]
We denote, for any weighted tree $t$ and $n\geq0$,
\[
  \ceffprt_n(t) = \Pq^t_\root (\tau^{(n)} < \tau_{\prt\root}),
\]
which is the conductance between $\prt\root$ and the $n$-th level of $t$.

Now, by the Markov property (or by the law of conductances in parallel),
for any $n\geq1$,
\begin{equation}
  \label{c_beta}
  \ceff_n(t) = \sum_{i=1}^{\numch[t]{\root}} \wt[t](i) \ceffprt_{n-1}(t[i]),
\end{equation}
while by the law of resistances in series,
\begin{equation}
  \label{beta_c}
  \ceffprt_n(t) = \frac{\ceff_n(t)}{1 + \ceff_n(t)}.
\end{equation}
Combining these identities gives, for all $n\geq2$,
\begin{equation}
  \label{eq:ceff2}
  \ceff_n(t) = \sum_{i=1}^{\numch[t]{\root}} \wt[t](i) \frac{\ceff_{n-1}(t[i])}{1 +
  \ceff_{n-1}(t[i])}.
\end{equation}
By the branching property, and the hypothesis \eqref{eq:hyp_normalized} we
already obtain the recursive equation:
\begin{equation}
  \label{eq:ceff1}
  \Et[\ceff_n] = \Et\bracks[\Big]{\sum_{i=1}^\nu
  A(i)}\Et\bracks[\Big]{\frac{\ceff_{n-1}}{1+\ceff_{n-1}}} =
  \Et\bracks[\Big]{\frac{\ceff_{n-1}}{1+\ceff_{n-1}}}.
\end{equation}

From now on, we let, for $n\geq1$,
\[
  u_n = \Et [\ceff_n] \quad \text{and} \quad a_n = 1 / u_n.
\]
Moreover, for any random variable $\xi$ such that $\Et[\xi]$ exists in
$\intoo{0}{\infty}$, we define the renormalized random variable $\renorm{\xi}$
by
$
  \renorm{\xi} = \xi/\Et[\xi].
$

Going back to \eqref{eq:ceff1}, we obtain
\begin{align*}
  u_n = \Et\bracks[\Big]{\frac{\ceff_{n-1}}{1+\ceff_{n-1}}}
  = u_{n-1} - \Et\bracks[\Big]{\frac{\ceff_{n-1}^2}{1+\ceff_{n-1}}}
= u_{n-1} -
u_{n-1}\Et\bracks[\Big]{\frac{\renorm{\ceff_{n-1}}^2}{a_{n-1}+\renorm{\ceff_{n-1}}}}.
\end{align*}
Introducing, for $a>0$ the (convex) function $\phi_a : x \mapsto
x^2/(a+x)$, the previous equality becomes
\begin{equation}
  \label{eq:recun}
  u_{n-1} - u_n = u_{n-1} \Et \bracks{\phi_{a_{n-1}}\renorm{\ceff_{n-1}}},
\end{equation}
which is key in the rest of this work.

The rough idea here, is that we expect $\renorm{\ceff_{n}}$ to be ``close''
to $M_\infty$ so that, as $n$ is large,
\[
  \Et \bracks{\phi_{a_{n-1}} \renorm{\ceff_{n-1}}} \approx
  \Et \bracks{\phi_{a_{n-1}} (M_\infty)}.
\]
Indeed, we will prove that at least one of the inequalities is correct in this
heuristics.
To do so, we need the following fact:
\begin{fact}
  \label{fact:ineqconvexrenorm}

  Let $\xi$ be a non-negative random
  variable such that $\E[\xi]$ is in $\intoo{0}{\infty}$.
  Let $\phi : \R_+ \to \R_+$ be a differentiable,
  convex function.
  Assume that
  \begin{equation}
    \E \bracks*{\phi\renorm{\xi}} < \infty \implies \E
    \bracks*{\phi((1+\varepsilon)\renorm{\xi})}
    < \infty \text{ for some $\varepsilon > 0$.}
    \label{eq:hypfactphi1}\tag{*}
  \end{equation}
  Then,
  \begin{equation*}
    \E \bracks*{ \phi\renorm*{\frac{\xi}{1+\xi}} }
    \leq
    \E \bracks*{\phi\renorm{\xi}}.
  \end{equation*}
\end{fact}
This fact itself is already stated in \cite[Proof of
Lemma~3.1]{hu_local_times_subdiff}
and is mostly a consequence of 
\cite[Formula~3.3]{hu_shi_subdiffusive}.
However, since the statement of this formula is rather long, we give a more
direct but less general proof in the appendix.
\begin{remark}
  Hypothesis~\eqref{eq:hypfactphi1} is satisfied whenever
  there exist $C > 0$ and $\varepsilon > 0$
  such that, for any $x$ large enough,
  \begin{equation}
    \phi((1 + \varepsilon)x) \leq C \phi(x).
    \label{eq:hypfactphi2}\tag{**}
  \end{equation}
  Furthermore, if $\phi : \R_+ \to \R_+$ is a convex, differentiable function
  satisfying \eqref{eq:hypfactphi2}, so is $x \mapsto \phi(ax + b)$, for any $a$ and $b$ in
  $\R_+$.
\end{remark}
\begin{lemma}
  \label{lem:ineq_phi}
  Let $\phi : \R_+ \to \R_+$ be \emph{any} convex,
  differentiable function satisfying \eqref{eq:hypfactphi2}.
  Then, for any $n\geq1$,
  \[
    \Et \bracks{\phi\renorm{\ceff_n}} \leq \Et \bracks{\phi(M_n)}.
  \]
  Moreover, whenever for all $x\geq0$, $\phi(x) \leq Cx^p$, for some constant
  $C>0$, for
  $p\in \intoo{1}{\kappa}$ if $\kappa \leq 2$, and for $p = 2$ if $\kappa > 2$, one
  has
  \[
    \Et \bracks{\phi\renorm{\ceff_n}} \leq \Et \bracks{\phi(M_n)}
    \leq \Et \bracks{\phi(M_\infty)}.
  \]
\end{lemma}
\begin{proof}[Proof of Lemma~\ref{lem:ineq_phi}]
  We prove by induction that, for any $n\geq1$, for any $\phi$ as in the
  first statement of the lemma,
  \[
    \Et[\phi\renorm{\ceff_n}] \leq \Et[\phi(M_n) ].
  \]
  Notice that $\ceff_1 (T) = M_1(T)$ and
  for $n\geq1$, by \eqref{eq:ceff1} and \eqref{eq:ceff2}, observe that
  \[
    \renorm{\ceff_{n+1}(T)} = 
    \sum_{i=1}^{\numch[]{\root}} \wt[](i)
    \renorm[\Big]{\frac{\ceff_n(T[i])}{1 + \ceff_n(T[i])}}
  \]
  By independence, it suffices to show that, for any $k\geq0$,
  for any $(a_1, a_2, ..., a_k) \in \intoo{0}{\infty}^k$, for any (convex,
  differentiable, \dots) function $\phi$,
  \[
    \Et \bracks[\Big]{\phi\pars[\Big]{\sum_{i=1}^{k} a_i
    \renorm[\Big]{\frac{\ceff_n(T[i])}{1 + \ceff_n(T[i])}}}}
    \leq
    \Et \bracks[\Big]{\phi\pars[\Big]{\sum_{i=1}^{k} a_i M_n(T[i])}}.
  \]
  Assume the result is true for $k\geq0$.
  Then,
  \begin{align*}
    &\Et \bracks[\Big]{\phi\pars[\Big]{\sum_{i=1}^{k} a_i
    \renorm[\Big]{\frac{\ceff_n(T[i])}{1 + \ceff_n(T[i])}}
    +
    a_{k+1}
    \renorm[\Big]{\frac{\ceff_n(T[k+1])}{1 + \ceff_n(T[k+1])}}
    }}
    \\
    &=
    \Et \bracks[\Big]{\phi\pars[\Big]{
    B+    a_{k+1}
    \renorm[\Big]{\frac{\ceff_n(T[k+1])}{1 + \ceff_n(T[k+1])}}
  }},
  \end{align*}
  where $B$ is the sum in the first expectation and is independent of the last
  term.
  Reasoning conditionally with respect to $B$,
  we may use the previous fact with the
  function
  $x \mapsto \phi(a_{k+1}x + B)$ to obtain
  that this expectation is bounded from above by
  \[
  \Et \bracks[\big]{\phi\pars[]{
    B+    a_{k+1}
    \renorm*{\ceff_n(T[k+1])}}
    }
    \leq
    \Et \bracks[\big]{\phi\pars[]{
      B+    a_{k+1}M_n(T[k+1])
  }},
  \]
  where for the last inequality, we used the induction hypothesis on $n$.
  Now reason conditionnally on $M_n(T[k+1])$ to use the induction hypothesis on
  $k$.

  For the last assertion, it suffices to see that
  $(M_n)$ is bounded in $L^p$ if $1< p <\kappa$ for $\kappa \leq 2$ and bounded
  in $L^2$ for $\kappa > 2$, which is certainly well-known.
  For a quick proof,
  let $n\geq 1$ and $p$ as above ($p=2$ if $\kappa > 2$). 
  Reason conditionally with respect to $\mathcal{F}_1$ and use
  an inequality due to Neveu (stated later in this paper as \eqref{eq:ineqneveu}):
  \[
    \Et\bracksof{M_{n+1}^p}{\mathcal{F}_1}
    \leq \sum_{i = 1}^{\numch[]{\root}} \wt[](i)^p \Et \bracks[]{M_n^p}
  + \pars[\Big]{\sum_{i = 1}^{\numch[]{\root}} \wt[](i)\Et\bracks[]{M_n}}^p.
  \]
  As a consequence,
  \[
    \Et [M_{n+1}^p] \leq e^{\psi(p)}\Et \bracks[]{M_n^p} +
    \Et \bracks[\Big]{\pars[\Big]{\sum_{i=1}^\nu A(i)}^p}
  \]
  Since $\psi(p)<0$ and
  $\Et \bracks[\big]{\pars[\big]{\sum_{i=1}^\nu A(i)}^p} < \infty$ by
  assumption,
  this is easily seen to imply that $\sup_{n\geq1} \Et [M_{n}^p] < \infty$.
\end{proof}

Now we need to study the asymptotics of $a \mapsto \Et
\bracks{\phi_{a} (M_\infty)}$ as $a$ goes to infinity.
For later use, consider in general, for $p>0$ and $a>0$,
\begin{equation}
\varphi_p(a) = \Et \bracks[\Big]{\pars[\Big]{\frac{M_\infty^2}{a +
M_\infty}}^p} = \Et \bracks[\big]{\phi_a(M_\infty)^p}.
\end{equation}

Using the tail probability estimate in Fact~\ref{fact:liutail}, for $1 <
\kappa \leq 2$, or the integrability of $M_\infty^2$ for $\kappa > 2$, one
obtains:
\begin{lemma}
  \label{lem:order_varphi}
  As $a$ goes to infinity,
  \begin{equation}
    \label{eq:order_varphi}
    \varphi_p(a)
    \begin{cases}
      \oforder a^{p-\kappa} & \text{if $\kappa / 2 < p < \kappa \leq 2$;}\\
      \oforder a^{p-\kappa}\log a & \text{if $p = \kappa / 2$;}\\
      \leq C a^{p-2} & \text{for some constant $C>0$, if $\kappa > 2$ and $1 < p < 2$;}\\
      \sim \Et [M_\infty^2] / a & \text{if $\kappa > 2$ and $p = 1$.}
    \end{cases}
  \end{equation}
\end{lemma}
\begin{proof}
  Write $\Pt_{M_\infty}$ for the distribution of $M_\infty$.
  Differentiate
  the function $x \mapsto \left(\frac{x^2}{a+x}\right)^p$, to obtain
  \[
    \varphi_p(a) = \int_{0}^\infty
    \pars[\Big]{ \int_0^{s} p \frac{x^2 + 2ax}{(a + x)^2}
    \pars[\Big]{\frac{x^2}{a + x}}^{p-1} \dd x}
    \Pt_{M_\infty}(\dd s).
  \]
  Using Tonelli's theorem together with the change of variable $y=x/a$
  yields
  \begin{equation}
    \label{eq:phip_integral}
    \varphi_p(a)
    =
    pa^{p-\kappa}\int_0^\infty
    \frac{y^{2p - \kappa - 1}(y+2)}{(1+y)^{p+1}}
    (ay)^\kappa \Pt (M_\infty > a y)
    \dd y.
  \end{equation}
  Let $f(y)$ be the integrand in the last equation.

  Now assume that $1 < \kappa \leq 2$ and  write $\underline{\ell}$
  (respectively $\overline{\ell}$) for the inferior (respectively superior) limit of
  $s^\kappa \Pt( M_\infty > s)$, as $s$ goes to infinity.
  Consider $\varepsilon > 0$ so
  small that $\underline{\ell} - \varepsilon > 0$.
  Let $N > 0$ be so large that
  \[
    \forall s \geq N, \qquad
    s^\kappa \Pt( M_\infty > s)
    \in \intoo{\underline{\ell} - \varepsilon}{\overline{\ell} + \varepsilon}.
  \]
  Assume that $a > N$.
  On the interval $\intoo{0}{N/a}$, dominating $\Pt(M_\infty > ay)$ by $1$
  yields
  \[
    f(y) \leq a^\kappa y^{2p-1} \max_{0\leq y \leq N/a} \frac{2+y}{(1+y)^{p+1}},
  \]
  so that in any case,
  \[
    pa^{p-\kappa}\int_0^{N/a} f(y) \dd y
    \leq
    \bracks[\Big]{p N^{2p} \max_{0\leq y \leq 1} \frac{2+y}{(1+y)^{p+1}}}
    a^{-p},
  \]
  which will be negligible.
  On the other hand, if $y$ is in the interval $\intfo{N/a}{\infty}$, then
  \begin{equation}
    f(y) \leq (\overline{\ell} + \varepsilon)
    \frac{y^{2p - \kappa - 1}(y+2)}{(1+y)^{p+1}}
  \end{equation}
  and similarly for the lower bound.
  Those bounds are integrable on $\intoo{0}{\infty}$
  if $p > \kappa / 2$ and in this case, we may
  conclude by applying the monotone convergence theorem.

  Now assume that $p = \kappa/2$.
  The bound above is still integrable at the neighborhood of $\infty$, but
  not at the neighborhood of $0$.
  As a consequence, the main contribution in the integral comes from the term
  \[
    \int_{N/a}^1 f(y) \dd y
    \leq (\overline{\ell} + \varepsilon) \int_{N/a}^1 y^{-1} \frac{2+y}{(1+y)^{p+1}} \dd y
    \oforder_{a \to \infty} \log(a),
  \]
  and the same is true for the lower bound.

  Finally, assume that $\kappa > 2$ and recall that in this case,
  by our hypotheses,
  $\Et\bracks{M_\infty^2}$ is finite, thus by Markov's inequality, for all $r >
  0$,
  $\Pt(M_\infty > r) \leq \Et\bracks{M_\infty^2}/r^2$. Now, if $1 < p < 2$,
  the rest of the computations is exactly the same as in the first point,
  whereas if $p = 1$, by dominated convergence,
  \[
    a\varphi_1(a) = \Et\bracks*{\frac{M_\infty^2}{1 + M_\infty/a}}
    \xrightarrow[a \to \infty]{} \Et\bracks*{M_\infty^2}.\qedhere
  \]
  
\end{proof}
Going back to \eqref{eq:recun} and using the two previous lemmas, we see that,
for some constant $C$ in $\intoo{0}{\infty}$ and any $n\geq2$,
\begin{equation}
  \label{eq:pre_lower_bound}
  u_{n-1} - u_n \leq
  \begin{cases}
    C u_{n-1} a_{n-1}^{1-\kappa} = C u_{n-1}^{\kappa}
    & \text{if $1 < \kappa < 2$;} \\
    C u_{n-1} a_{n-1}^{-1} \log a_{n-1} = C u_{n-1}^2 \log (1 / u_{n-1}) 
    & \text{if $\kappa = 2$;} \\
    \Et[M_\infty^2] u_{n-1}^{2} & \text{if $\kappa > 2$.}
  \end{cases}
\end{equation}

To obtain our lower bound, it suffices to use one part of the
following elementary analysis
lemma:
\begin{lemma}
  \label{lem:un}
  Let $(u_n)$ be a non-increasing sequence of positive real numbers going to $0$
  as $n$ goes to infinity. Let $\alpha > 1$ and $C \in \intoo{0}{\infty}$.
  \begin{enumerate}
    \item If for $n$ large enough, $u_n - u_{n+1} \leq C u_n^{\alpha}$, then
      $
        \liminf_{} n^{\frac{1}{\alpha - 1}} u_n \geq [C(\alpha - 1)]^{-\frac{1}{\alpha - 1}}.
      $
    \item If for $n$ large enough, $u_n - u_{n+1} \geq C u_n^{\alpha}$, then
      $
        \limsup_{} n^{\frac{1}{\alpha - 1}} u_n \leq [C(\alpha - 1)]^{-\frac{1}{\alpha - 1}}.
      $
    \item If for $n$ large enough, $u_n - u_{n+1} \leq Cu_n^2 \log(1/u_n)$, then
      $\liminf u_n n \log n \geq C^{-1}$.
    \item If for $n$ large enough, $u_n - u_{n+1} \geq Cu_n^2 \log(1/u_n)$, then
      $\limsup u_n n \log n \leq C^{-1}$.
  \end{enumerate}
\end{lemma}
\begin{proof}
  To prove the first two assertions of the lemma, consider, for $x > 0$,
  $f(x) = x^{1 - \alpha} / (\alpha - 1)$.
  By the mean value theorem,
  \[
    f(u_{n+1}) - f(u_{n}) = f'(\xi_n) (u_{n+1} - u_{n}) =
    \xi_n^{-\alpha}(u_{n}-u_{n+1}),
  \]
  for some $\xi_n \in \intoo{u_{n+1}}{u_{n}}$.
  In case~1, it is easy to see that $u_n \sim u_{n+1}$, therefore
  \[
    f(u_{n+1}) - f(u_n) \leq C u_n^{\alpha} \xi_n^{-\alpha} \longrightarrow C,
  \]
  and we may conclude by averaging this inequality.

  In case~2,
  since $\xi_n \leq u_n$,
  \[
    f(u_{n+1}) - f(u_n) \geq C u_n^{\alpha} \xi_n^{-\alpha} \geq C.
  \]

  The method for the last two assertions is similar: we apply the mean value
  theorem to the
  function $g$ defined by
  \[
    g(x) = \frac{1/x}{\log (1/x)},
  \]
  whose derivative is
  \[
    g'(x) = - \frac{1}{x^2 \log (1/x)} \pars[\Big]{1 -
    \frac{1}{\log(1/x)}}.
  \]
  This gives, for instance in case~3,
  \[
    \liminf u_n n \log (1/u_n) \geq C^{-1},
  \]
  so that, for any large enough $n$,
  \[
    \log u_n + \log n + \log \log (1/u_n) \geq -\log(2C),
  \]
  thus
  \[
    \frac{\log n}{\log (1/u_n)} \geq 1 - \frac{\log (2C) +
    \log\log(1/u_n)}{\log(1/u_n)} \longrightarrow 1
  \]
  from which we conclude that $\liminf \log n / (\log (1/u_n)) \geq 1$, which
  ends the proof in case~3.
  The proof in case~4 is the same.
\end{proof}

Using the first and third points of this lemma with \eqref{eq:pre_lower_bound}
finally yields the following lower
bounds:
\begin{proposition}
  \label{prop:lower}
  Under the hypotheses \eqref{eq:hyp_normalized}, \eqref{eq:hyppsip1} and
  \eqref{eq:hypkappa},
  \begin{enumerate}
    \item $\liminf_{n\to\infty} n^{1/(\kappa-1)}\Et \bracks*{\ceff_n} > 0$, if
      $1 < \kappa < 2$;
    \item $\liminf_{n\to\infty} n\log n\Et \bracks*{\ceff_n} > 0$, if
      $\kappa = 2$ and
    \item $\liminf_{n\to\infty} n\Et\bracks*{\ceff_n} \geq
      1/\Et\bracks{M_\infty^2}$, if $\kappa > 2$.
  \end{enumerate}
\end{proposition}

\begin{remark}
  If we assume that we are in the non-lattice case
  these lower bounds can also be made explicit (in terms on
  the distribution of $M_\infty$) in the cases
  $1 < \kappa < 2$ and $\kappa = 2$.
  However, since our method does not provide explicit upper bounds,
  we have chosen not to make this additional assumption.
\end{remark}
\section{Upper bound and almost sure convergence}
We start with an easy \textit{a priori} upper bound.
\begin{lemma}\label{lem:aprioriupper}
  In any case, one has
\begin{equation}
  \label{eq:ineqapriori}
  \limsup_{n \to \infty} n u_n \leq 1.
\end{equation}
\end{lemma}
\begin{proof}
  Let $n \geq 2$.
  We go back to \eqref{eq:ceff1} but this time, we write
  \[
  u_{n-1} - u_n = \Et \bracks[\Big]{\frac{\ceff_{n - 1}^2}{1 + \ceff_{n-1}}}
\geq \Et \bracks[\Big]{\pars[\Big]{\frac{\ceff_{n-1}}{1 + \ceff_{n -1}}}^2}
\geq \pars[\Big]{\Et \bracks[\Big]{\frac{\ceff_{n - 1}}{1 + \ceff_{n - 1}}}}^2
    = u_n^2.
  \]
  This implies that
  \[
    \frac{1}{u_n} - \frac{1}{u_{n-1}} \geq \frac{u_n}{u_{n-1}} = 1 - 
    \Et \bracks{\phi_{a_{n-1}}\renorm{\ceff_{n-1}}},
  \]
  by the identity~\eqref{eq:recun}.
  Using Lemma~\ref{lem:ineq_phi} yields
  \[
    \frac{1}{u_n} - \frac{1}{u_{n-1}} \geq
    1 - \varphi_1(a_{n-1}).
  \]
  Since this lower bound goes to $1$ as $n$ goes to infinity,
  averaging the previous inequality and using Ces\`aro's lemma yields
  \[
    \liminf_{n\to\infty} \frac{1}{nu_n} \geq 1,
  \]
  hence the result.
\end{proof}

To obtain more refined bounds, we first iterate \eqref{eq:ceff2}. Using
repeatedly the
identity,
\[
  \frac{x}{1+x} = x - \frac{x^2}{1+x} = x - \phi_1(x), \quad \forall x > 0,
\]
we obtain that, for any $n > k$,
\begin{align*}
  \ceff_n(T) &= \sum_{\abs{x}=1} \cd[](x) \ceff_{n-1}(T[x])
- \sum_{\abs{x}=1} \cd[](x)\phi_1(\ceff_{n-1} (T[x])),
    \\
  &= \sum_{\abs{x}=k} \cd[](x) \ceff_{n-k}(T[x])
    - \sum_{\abs{x}\leq k} \cd[](x)\phi_1( \ceff_{n-\abs{x}}(T[x]) ).
\end{align*}
Dividing by $u_n = \Et[\ceff_n]$, and using the equality
\begin{equation*}
    \phi_1( \ceff_{n-\abs{x}}(T[x]) )
  =
    \frac{u_{n-\abs{x}}^2 \renorm{\ceff_{n-\abs{x}}(T[x])}}{%
    u_{n-\abs{x}}(a_{n-\abs{x}} + \renorm{\ceff_{n-\abs{x}}(T[x])})}
  = \phi_{a_{n-\abs{x}}} \renorm{\ceff_{n-\abs{x}}(T[x])},
\end{equation*}
we finally obtain
\begin{equation}
  \label{eq:renormrec}
  \renorm{\ceff_n(T)}
  =
  \frac{u_{n-k}}{u_n}
  \sum_{\abs{x}=k} \cd[](x) \renorm{\ceff_{n-k}(T[x])}
  - \frac{1}{u_n}\sum_{\abs{x}\leq k}  \cd[](x)
  u_{n-\abs{x}}
  \phi_{a_{n-\abs{x}}} \renorm{\ceff_{n-\abs{x}}(T[x])} .
\end{equation}
On the other hand, by definition of $M_n$,
\begin{equation}
  \label{eq:martrec}
  M_n(T) = \sum_{\abs{x}=k} \cd[](x)M_{n-k} (T[x]),
\end{equation}
therefore, for any $n > k$,
\begin{equation}
  \label{eq:ceffXY}
  \renorm{\ceff_n(T)}
  = \frac{a_n}{a_{n-k}}M_\infty(T)
  + \frac{a_n}{a_{n-k}}X_{k,n}(T)
  - \sum^{k}_{j=1} \frac{a_n}{a_{n-j}} Y_{j,n}(T),
\end{equation}
where
\begin{equation*}
  X_{k,n} (T)=
  \sum_{\abs{x}=k} \cd[](x) \left( \renorm{\ceff_{n-k}(T[x])} - M_\infty(T[x])\right),
\end{equation*}
and for $j<n$,
\begin{equation*}
  Y_{j,n} (T)
  =
  \sum_{\abs{x} = j}
  \cd[](x)
  \phi_{a_{n-j}}\renorm{\ceff_{n-j}(T[x])}.
\end{equation*}
Our next step is, for $p\in \intoo{1}{\kappa\mi2}$,
to estimate the $L^p$ norms of $X_{k,n}$ and $Y_{j,n}$ in order
to prove the convergence in $L^p$ of $\renorm{\ceff_n}$ towards $M_\infty$.
Remark that, by the branching property,
conditionally on $\mathcal{F}_k$, $X_{k,n}$ is a sum of
independent, \emph{centered} variables while conditionally on $\mathcal{F}_j$,
$Y_{j,n}$ is a sum of independent, \emph{non-negative} random variables.
We may therefore use the following upper bounds:
\begin{fact}
  Let $p$ be a real number in $\intff{1}{2}$ and assume that $\xi_1, \dotsc,
  \xi_k$ are independent real-valued random variables such that for all $1 \leq
  i \leq k$, $\E[\abs{\xi_i}^p] < \infty$.
  \begin{enumerate}
    \item If $\xi_1, \dotsc, \xi_k$ are non-negative, then
      \begin{equation}
        \label{eq:ineqneveu}
        \E[ (\xi_1 + \dotsb + \xi_k)^p ] \leq
        \sum^{k}_{i=1} \E \bracks*{\xi_i^p} +
      \pars[\bigg]{ \sum_{i=1}^k \E \xi_i }^p.
      \end{equation}
    \item If $\xi_1, \dotsc, \xi_k$ are centered, then
      \begin{equation}
        \label{eq:ineq_von_bahr}
        \E \bracks*{\abs*{\xi_1 + \dotsb + \xi_k }^p}
        \leq 2 \sum^{k}_{i=1} \E \bracks*{\abs*{\xi_i}^p}.
      \end{equation}
  \end{enumerate}
\end{fact}
The first inequality is due to Neveu (\cite{neveu_multiplicative})
while the second is borrowed from von Bahr and Esseen (\cite{vonbahr1965}, see also
\cite[p.~83]{petrov_limit}).
\begin{lemma}\label{lem:lpbounds}
  Let $p$ be in $\intoo{1}{\kappa\wedge 2}$. Let $1 \leq k \leq
  n$, then, in any case,
  \begin{gather}
    \label{eq:lpxkn}
    \lnorm{p}{X_{k,n}} \leq 2^{1+1/p} \lnorm{p}{M_\infty} e^{k\psi(p)/p};
    \\
    \label{eq:lpyjn}
    \lnorm{p}{Y_{j,n}} \leq e^{j\psi(p)/p} \varphi_p(a_{n-j})^{1/p}
    + \lnorm{p}{M_\infty}\varphi_1(a_{n-j}).
  \end{gather}
\end{lemma}
\begin{proof}
  For $X_{k,n}$,
  we apply, conditionally on $\mathcal{F}_k$, the inequality
  \eqref{eq:ineq_von_bahr}, to obtain
  \begin{equation}
    \label{eq:xkn_cond}
    \Et\bracksof*{\abs{X_{k,n}}^p}{\mathcal{F}_k}
    \leq 2 \sum_{\abs{x}=k}\cd[](x)^p
    \Et\bracks*{\abs{\renorm{\ceff_{n-k}}- M_\infty}^p}.
  \end{equation}
  Now recall that, by convexity and Lemma~\ref{lem:ineq_phi},
  $\Et\bracks*{\renorm{\ceff_n}^p} \leq \Et\bracks*{M_\infty^p}$,
  therefore,
  \[
    \Et \pars[\big]{\abs{ \renorm{\ceff_{n-k}} - M_\infty}}^p
    \leq
    2^{p-1} \pars[\big]{ \Et \bracks[\big]{\renorm{\ceff_{n-k}}^p}
    + \Et \bracks[\big]{M_\infty^p}
    }
    \leq 2^p \Et \bracks[\big]{M_\infty^p}.
  \]
  On the other hand,
  \begin{align*}
    \Et \bracks[\Big]{\sum_{\abs{x}=k} \cd[](x)^p}
    &= \Et \bracks[\Big]{\Et \bracksof[\Big]{%
    \sum_{\abs{y}=k-1} \cd[](y)^p
   \pars[\big]{
     \textstyle\sum_{i=1}^{\numch[]{y}}
     \wt[](yi)^p}}{\mathcal{F}_{k-1}}}
    \\
    &=e^{\psi(p)} \Et \bracks[\Big]{\sum_{\abs{x}=k-1} \cd[](x)^p}
    = \dotsb = e^{k\psi(p)}.
   \end{align*}
   Taking the expectation on both sides of \eqref{eq:xkn_cond} yields
  \begin{equation*}
    \Et \bracks[\big]{\abs{X_{k,n}^p}} \leq 2^{p+1} e^{k\psi(p)}
    \Et \bracks[\big]{M_\infty^p},
  \end{equation*}
  hence our inequality.

  For $Y_{j,n}$, we use the inequality
  \eqref{eq:ineqneveu} conditionally on $\mathcal{F}_j$:
  \begin{equation*}
    \Et \bracksof[\big]{Y_{j,n}^p}{\mathcal{F}_j}
    \leq \sum_{\abs{x}=j} \cd[](x)^p
    \Et \bracks[\big]{\pars[\big]{\phi_{a_{n-j}}\renorm{\ceff_{n-j}}}^p}
  + \pars[\Big]{\sum_{\abs{x}=j}\cd[](x) \Et \phi_{a_{n-j}} \renorm{\ceff_{n-j}}}^p
  \end{equation*}
  Therefore, using twice Lemma~\ref{lem:ineq_phi},
  \begin{equation*}
      \Et\bracks*{{Y_{j,n}}^p}
      \leq e^{j\psi(p)} \varphi_p(a_{n-j}) + \varphi_1(a_{n-j})^p \Et [M_j^p],
  \end{equation*}
  which implies our inequality.
\end{proof}

We will apply \eqref{eq:ceffXY} by properly choosing $k$ according to $n$.
To this end, we need the following lemma.
\begin{lemma}
  \label{lem:anoveran-k}
  Let $(k_n)_{n\geq1}$ be a sequence of non-negative integers.
  If $\lim_{n\to\infty} k_n/n = 0$, then $a_{n-k_n} \sim a_n$.
\end{lemma}
\begin{proof}
  Recall that, by \eqref{eq:recun} and \Cref{lem:ineq_phi}, for any $n \geq 2$,
  \[
    1\geq\frac{u_n}{u_{n-1}} = 1 - \Et \phi_{a_{n-1}}\renorm{\ceff_{n-1}}
    \geq 1 - \varphi_1(a_{n-1}).
  \]
  Iterating this inequality yields, for any large enough $n$,
  \[
    1 \geq \frac{u_n}{u_{n-k_n}} \geq \prod_{i=1}^{k_n}
  \pars[\big]{1 - \varphi_1(a_{n-i})}.
  \]
  Therefore, since the sequence $(a_n)$ is non-decreasing and the function $\varphi_1$ is
  non-increasing,
  \[
    \frac{u_n}{u_{n-k_n}} \geq \pars[\big]{1 - \varphi_1(a_{n-k_n})}^{k_n}.
  \]
  On the other hand, combining the lower bounds for $(u_n)$
  (Proposition~\ref{prop:lower}) with \eqref{eq:order_varphi} shows that, in any
  case, we may find $C \in \intoo{0}{\infty}$ such that for all $n \geq 1$,
  \[
    \varphi_1(a_n) \leq \frac{C}{n}.
  \]
  Plugging this into the previous inequality gives
  \[
  \frac{u_n}{u_{n-k_n}} \geq \pars[\big]{1 - \frac{C}{n - k_n}}^{k_n}
  \xrightarrow[ n\to\infty ]{  } 1. \qedhere{}
  \]
\end{proof}
\begin{lemma}
  \label{lem:cvlp}
  For any $p$ in $\intoo{1}{\kappa\mi2}$, the sequence $(\renorm{\ceff_n})$
  converges towards $M_\infty$ in $L^p$.
\end{lemma}
\begin{proof}
  Recall that $\psi(p) < 0$.
  For $n \geq 1$, let
  $k_n = \floor{ (-2/\psi(p)) \log a_n}$
  so that, by \Cref{lem:lpbounds}, for some constant $C_1 > 0$,
  \[
    \lnorm{p}{X_{k_n,n}} \leq C_1 a_n^{-2/p}.
  \]
  It is clear from Proposition~\ref{prop:lower} that $(k_n)$ satisfies the
  assumption of the previous lemma.
  Moreover,
  \begin{align*}
    \sum_{j=1}^{k_n} \lnorm{p}{Y_{j,n}}
    &\leq \frac{e^{\psi(p)/p}}{1-e^{\psi(p)/p}}\varphi_p(a_{n-k_n})^{1/p}
    + \lnorm{p}{M_\infty}k_n\varphi_1(a_{n-k_n})
    \\
  &\leq C_2 \pars[\big]{\varphi_p(a_{n-k_n})^{1/p} + k_n\varphi_1(a_{n-k_n})},
  \end{align*}
  for some constant $C_2 > 0$.
  Now, using \eqref{eq:order_varphi}, we see that, for some constants
  $C_3$, $C_3'$, $C_4$ and $C_4'$ in
  $\intoo{0}{\infty}$, in any case, for any $n\geq1$,
  \begin{align*}
    k_n \varphi_1(a_{n-k_n}) \leq C_3 \log (a_{n-k_n})^2 a_{n-k_n}^{1 - \kappa\mi2}
    &\leq C_3' \log (a_{n})^2 a_{n}^{1 - \kappa\mi2}
    \\
    \text{and} \quad \varphi_p(a_{n-k_n})^{1/p} \leq C_4 a_{n-k_n}^{(p-\kappa\mi2)/p}
    &\leq C'_4 a_n^{1 - (\kappa\mi2)/p}.
  \end{align*}
  Since
  $
    -2/p < -1 \leq 1 - \kappa\mi2 < 1 - (\kappa\mi2)/p < 0,
  $
  there exists $C > 0$ such that, for any $n\geq1$,
  \begin{equation}
    \label{eq:lpnormsXY}
    \lnorm{p}{X_{k_n,n}} + \sum_{j=1}^{k_n} \lnorm{p}{Y_{j,n}}
    \leq C a_n^{1-(\kappa\mi2)/p}.
  \end{equation}
  To conclude this proof, it remains to see that,
  by \eqref{eq:ceffXY} and Minkowski's inequality,
  \[
    \lnorm{p}{\renorm{\ceff_n} - M_\infty}
    \leq
    \pars[\Big]{\frac{a_n}{a_{n-k_n}} - 1} \lnorm{p}{M_\infty}
    + \frac{a_n}{a_{n-k_n}} \lnorm{p}{X_{k_n,n}}
    + \frac{a_n}{a_{n-k_n}} \sum_{j=1}^{k_n} \lnorm{p}{Y_{j,n}}
  \]
  By our choice of $(k_n)$ and the preceding lemma, this upper bound is
  asymptotically equivalent to $\lnorm{p}{X_{k_n,n}}
    + \sum_{j=1}^{k_n} \lnorm{p}{Y_{j,n}}$, which goes to $0$ as $n$ goes to
    infinity.
\end{proof}

From there, the almost sure convergence of $(\renorm{\ceff_n})$ towards
$M_\infty$ can be classically obtained by accelerating this $L^p$ convergence
and using a monotony argument.
\begin{proof}[Proof of the almost sure convergence in Theorem~\ref{thm}]
  Using \eqref{eq:lpnormsXY}, together with the \textit{a priori} bound
  \eqref{eq:ineqapriori}, shows the existence of $C > 0$ and $C'>0$
  such that for any
  $n\geq1$,
  \[
    \Et\bracks[\big]{\abs{\renorm{\ceff_n} - M_\infty}^p}
    \leq C' a_{n}^{p - \kappa\mi2} \leq \frac{C}{n^{\kappa\mi2 - p}}.
  \]
  Letting $\alpha = \ceil{2/(\kappa\mi2 - p)}$, we obtain that
  \[
    \sum_{n\geq1} \Et\bracks[\big]{\abs{\renorm{\ceff_{n^\alpha}} - M_\infty}^p}
    <\infty,
  \]
  hence, by Borel-Cantelli's lemma, $(\renorm{\ceff_{n^\alpha}})$ converges
  almost surely to $M_\infty$.
  
  Now, let, for $n\geq 1$, $r_n = \ceil{n^{1/\alpha}}$. Then, for all $n\geq 1$,
  \[
    (r_n - 1)^\alpha \leq n \leq r_n^\alpha,
  \]
  and by the fact that the sequence $\left(\ceff_n\right)$ is non-increasing,
  \[
    \ceff_{r_n^\alpha} \leq \ceff_n \leq \ceff_{(r_n - 1)^\alpha}.
  \]
  This implies that
  \[
    \frac{u_{r_n^\alpha}}{u_{(r_n-1)^\alpha}}
    \renorm{\ceff_{r_n^\alpha}}
    \leq \renorm{\ceff_n}
    \leq
    \frac{u_{(r_n-1)^\alpha}}{u_{r_n^\alpha}}
    \renorm{\ceff_{(r_n-1)^\alpha}}.
  \]
  Now, write $(r_n-1)^\alpha = r_n^\alpha - s_n$. Since $s_n / r_n^\alpha \to 0$,
  we may use Lemma~\ref{lem:anoveran-k} to see that
  \[
    \frac{u_{r_n^\alpha}}{u_{(r_n-1)^\alpha}} \xrightarrow[n\to\infty]{}
    1,
  \]
  which concludes this part of the proof.
\end{proof}
We may now conclude in the case $\kappa > 2$.
\begin{proof}[End of the proof of Theorem~\ref{thm} in the case $\kappa > 2$]
  We already know that, for all $n\geq 1$, by Lemma~\ref{lem:ineq_phi},
  $\Et\bracks{\renorm{\ceff_n}^2} \leq \Et\bracks{M_\infty^2}$.
  Now, by the almost-sure convergence of $\renorm{\ceff_n}$ to $M_\infty$ and
  Fatou's lemma,
  $\Et\bracks{M_\infty^2}\leq \liminf \Et\bracks{\ceff_n^2}$, thus
  $\Et\bracks{\renorm{\ceff_n}^2} \to \Et\bracks{M_\infty^2}$.

  Finally, by dominated convergence,
  \[
    \Et\bracks*{\frac{\renorm{\ceff_{n-1}^2}}{a_{n - 1} + \renorm{\ceff_{n-1}}}}
    \sim u_n \Et\bracks*{M_\infty^2},
  \]
  and by the identity \eqref{eq:recun} and Lemma~\ref{lem:un},
  \[
    u_n \sim \frac{1}{n\Et\bracks[\big]{M_\infty^2}}.
  \]
  To obtain the last equality, proceed in the following way:
  \begin{align*}
    \Et \bracks[\big]{M_\infty^2}
    &=
    \Et \bracks[\big]{\sum_{i=1}^{\numch[]{\root}} \wt[](i)^2
    M_\infty(T[i])^2} +
    \Et \bracks[\Big]{ \sum_{1 \leq i \neq j \leq \numch[]{\root}}^{}
    \wt[](i)\wt[](j)M_\infty(T[i])M_\infty(T[j]) }
    \\
  &= \Et \bracks[\big]{\sum_{i=1}^{\nu} A(i)^2}
  \Et \bracks[\big]{M_\infty^2} + \Et\bracks[\big]{\sum_{1 \leq i \neq j \leq \nu}^{} A(i)A(j)
  },
\end{align*}
by the branching property. Finally notice that in the case $\kappa > 2$,
$
  \Et \bracks[\big]{\sum_{i=1}^{\nu} A(i)^2} = e^{\psi(2)} < 1.
$
\end{proof}
To handle the case $1 < \kappa \leq 2$, we need a uniform lower bound on the
tail probability of $\renorm{\ceff_n}$.
\begin{lemma}
  \label{lem:tailceff}
  If $\kappa \in \intof{1}{2}$, we can find $\delta_0 > 0$ and $c_0 >
  0$ such that
  \[
    \Pt \pars*{\renorm*{\ceff_n} > r}
    \geq c_0 r^{-\kappa}, \qquad \forall r \in \intff{1}{\delta_0 a_n},\ \forall
    n \geq 1.
  \]
\end{lemma}
\begin{proof}
  Let $\delta > 0$ and $r \geq 1$. Let $(k_n)$ be as in the proof of
  Lemma~\ref{lem:cvlp}.
  Fix some $p \in \intoo{1}{\kappa}$.
  By \eqref{eq:ceffXY},
  \begin{align*}
    M_\infty &= \frac{a_{n-k_n}}{a_n}\renorm{\ceff_n} - X_{k_n,n} +
    \sum_{j=1}^{k_n} \frac{a_{n-k_n}}{a_{n-j}} Y_{j,n} \\
             &\leq
             \frac{a_{n-k_n}}{a_n}\renorm{\ceff_n} + \abs{X_{k_n, n}}
             + \sum_{j=1}^{k_n} Y_{j,n},
  \end{align*}
  so by the union bound,
  \[
    \Pt(\renorm{\ceff_n} > r)
    \geq
    \Pt\pars[\Big]{\frac{a_{n-k_n}}{a_n}\renorm{\ceff_n} > r}
    \geq
    \Pt \pars{ M_\infty > 2r}
    -
  \Pt \pars[\Big]{\abs*{X_{k_n,n}} + \sum_{j=1}^{k_n} Y_{j, n} > r}.
  \]
  By Markov's inequality and then the inequality~\eqref{eq:lpnormsXY},
  \begin{align*}
  \Pt \pars[\Big]{\abs*{X_{k_n,n}} + \sum_{j=1}^{k_n} Y_{j, n} > r}
      &\leq r^{-p} 
  \pars[\Big]{\lnorm{p}{X_{k_n,n}} + \lnorm[\Big]{p}{\sum_{j=1}^{k_n} Y_{j, n}}}^p
    \\
    &\leq C_1 a_n^{p-\kappa} r^{-p},
    \end{align*}
    for some constant $C_1 \in \intoo{0}{\infty}$.

    On the other hand, by Fact~\ref{fact:liutail},
    \[
      \inf_{r\geq 1} r^\kappa \Pt(M_\infty > 2r) \eqqcolon C_2 > 0.
    \]
    This implies that, for all $r$ in $\intff{1}{\delta a_n}$,
    \[
      r^\kappa\Pt(\renorm{\ceff_n} > r)
      \geq C_2 - C_1r^{\kappa-p}a_n^{p-\kappa} \geq C_2 - C_1 \delta^{\kappa
      -p},
    \]
    which is positive as soon as $\delta$ is small enough.
\end{proof}
\begin{proof}[End of the proof of Theorem~\ref{thm} : upper bounds]
  Here, we assume that $1<\kappa\leq2$.
  Recall that, for any $n\geq2$,
  \[
    u_{n-1} - u_n =
    u_{n-1}
    \Et \bracks*{ \frac{\renorm{\ceff_{n-1}}^2}{a_{n-1} + \renorm{\ceff_{n-1}}}}.
  \]
  By \Cref{lem:tailceff} and the same computation as those yielding
  \eqref{eq:phip_integral},
  \begin{equation*}
  \Et \bracks[\Big] {\frac{\renorm{\ceff_{n-1}}^2}{a_{n-1} + \renorm{\ceff_{n-1}}}}
    \geq c_0 \int_1^{\delta_0a_{n-1}}
    \frac{x^2 + 2a_{n-1} x}{(a_{n-1} + x)^2}
    x^{-\kappa} \dd x.
  \end{equation*}
  Thus the change of variable $y=x/a_{n-1}$ leads to
  \begin{equation*}
  \Et \bracks[\Big] {\frac{\renorm{\ceff_{n-1}}^2}{a_{n-1} + \renorm{\ceff_{n-1}}}}
    \geq c_0 a_{n-1}^{1-\kappa}
    \int_{1/a_{n-1}}^{\delta_0}
    y^{1-\kappa} \frac{y+2}{(1+y)^2} \dd y.
  \end{equation*}
  This integral converges in the case $\kappa < 2$ while in the case $\kappa =
  2$, it becomes larger than some constant times $\log a_{n-1}$ for $n$ large enough .
  Hence, there exists $C > 0$ such that, for $n\geq2$,
  \[
    u_{n-1} - u_n \geq C 
    \begin{cases}
      u_{n-1} a_{n-1}^{1-\kappa} = u_{n-1}^{-\kappa} & \text{if $1 < \kappa <
      2$} \\
      u_{n-1} a_{n-1}^{-1} \log a_{n-1}
        = u_{n-1}^2\log (1/u_{n-1}) &
        \text{if $\kappa = 2$,}
    \end{cases}
  \]
  and we may conclude by Lemma~\ref{lem:un}.
\end{proof}
\section*{Appendix : proofs omitted from the main text}

\begin{fact}[Null recurrence]
  If \eqref{eq:hyp_normalized} holds, then for
  $\GW$-almost every tree $t$,
  the random walk on $t$
  of probability kernel $\Pq^t$ is recurrent.
  If, additionnally,
  \eqref{eq:hyppsip1} and
  \eqref{eq:hyp_bigg} hold, then, for 
  $\GW$-almost every infinite tree $t$,
  the random walk on $t$
  of probability kernel $\Pq^t$ is null-recurrent.
\end{fact}
\begin{proof}
  For a weighted tree $t$, let
  $\ceffprt(t) = \Pq_\root^t(\tau_{\prt\root} = \infty)$
  and $\ceff(t) = \Pq_\root(\tau_{\root}^+ = \infty)/\Pq_\root^t(\root,\prt\root)$.
  These are the conductances between, respectively, $\prt\root$ and infinity,
  and $\root$ and infinity.
  By the Markov property,
  \[
    \ceffprt(t) = \frac{\ceff(t)}{1 + \ceff(t)}
    \quad\text{and}\quad
    \ceff(t) = \sum^{\nu_t(\root)}_{i=1} \wt(i)\ceffprt(t[i]).
  \]
  Now if $T$ is a weighted Galton-Watson tree and
  $\Et\bracks{\sum_{i=1}^{\nu} A(i)} = 1$, taking the expectation in
  the previous identities leads to
  \[
    \Et[\ceff(T)] = \Et[\ceffprt(T)] = \Et\bracks*{\frac{\ceff(T)}{1 + \ceff(T)}},
  \]
  which implies that, almost surely, $\ceff(T) = 0$, so the random walk is
  recurrent.

  To prove that it is null-recurrent, consider, for any recurrent
  weighted tree $t$,
  $\alpha(t) = \Eq_\root^t\bracks{\tau_{\prt\root}}$. We want to show that, almost
  surely on the event of non-extinction,
  $\alpha(T) = \infty$.
  The function $\alpha$ satisfies, by the Markov
  property,
  \[
    \alpha(t) = \Pq_\root^t(\root,\prt\root) +
    \sum_{i=1}^{\numch[t]{\root}} \Pq_\root^t(\root,i) (1 + \alpha(t[i]) + \alpha(t)).
  \]
  Thus we see that, if $\alpha(t)$ is finite, so are the $\alpha(t[x])$ for $x$ in
  $t$. In this case, one has
  \[
    \alpha(t) = 1 + \sum_{i=1}^{\numch[t]{\root}} \wt(i) + \sum_{i=1}^{\numch[t]{\root}}
    \wt(i)\alpha(t[i]),
  \]
  and iterating the previous identity,
  \[
    \alpha(t) = 1 + 2 \sum_{k=1}^{n - 1} M_k(t) + M_n(t)
    + \sum_{\abs{x}=n} \cd[t](x) \alpha(t[x])
    \geq \sum_{k=1}^n M_k(t),
    \quad \text{for all $n \geq 1$.}
  \]
  This shows that
  \[
    \Pt(\alpha(T) < \infty) \leq \Pt\pars[\Big]{\sum_{k=1}^\infty M_k(T) <
    \infty},
  \]
  but our assumptions and Biggins' theorem imply that, almost surely
  on the event of non-extinction, $M_n(T) \to M_\infty(T) > 0$, thus
  $\Pt(\alpha(T) < \infty)$ is the probability that $T$ is finite.
\end{proof}
Finally, we recall and prove Fact~\ref{fact:ineqconvexrenorm}.
\begin{fact*}
  Let $\xi$ be a non-negative random
  variable such that $\E[\xi]$ is in $\intoo{0}{\infty}$.
  Let $\phi : \R_+ \to \R_+$ be a differentiable,
  convex function.
  Further assume that
  \begin{equation}
    \E \bracks*{\phi\renorm{\xi}} < \infty \implies \E
    \bracks*{\phi((1+\varepsilon)\renorm{\xi})}
    < \infty \text{ for some $\varepsilon > 0$.}
    \label{eq:hypfactphi}\tag{*}
  \end{equation}
  Then,
  \begin{equation*}
    \E \bracks*{ \phi\renorm*{\frac{\xi}{1+\xi}} }
    \leq
    \E \bracks*{\phi\renorm{\xi}}.
  \end{equation*}
\end{fact*}
\begin{proof}
  We may assume that $\E[\phi\renorm{\xi}]$ is finite, otherwise there is
  nothing to prove.
  For $x\in \intff{0}{1}$ and $y\geq 0$, define
  \[
    f(x) = \E \bracks*{\phi\renorm*{\frac{\xi}{1+x\xi}}},
    \quad
    g(x) = \E \bracks[\Big]{\frac{\xi}{1+x\xi}},
    \quad
    h(x,y) = \frac{y}{1+xy},
    \quad
    \varphi(x,y) = \phi \pars[\Big]{\frac{h(x,y)}{g(x)}}.
  \]
  Notice that
  \[
    \frac{y}{1+y} \leq h(x,y) \leq \min \pars[\Big]{y, \frac{1}{x}}
    \quad\text{and}\quad
    \frac{\partial h}{\partial x} (x,y) = - h(x,y)^2,
  \]
  As a consequence,
  $g$ is finite and continuous on $\intff{0}{1}$, differentiable on
  $\intof{0}{1}$.

  Now let $\alpha \in \intoo{0}{1}$ and remark that, since $g$ is
  non-increasing,
  \begin{equation}
    \frac{h(x,\xi)}{g(x)}
    \leq
    \begin{dcases}
    \renorm{\xi}\frac{g(0)}{g(\alpha)}
    & \text{if $x \in \intff{0}{\alpha}$;}
    \\
    \frac{1}{\alpha g(1)}
    & \text{if $x \in \intff{\alpha}{1}$.}
    \end{dcases}
    \label{eq:ineqhoverg}
  \end{equation}
  In particular, the second inequality shows that $f$ is finite and continuous on $\intof{0}{1}$.

  Elementary calculus shows that
  \begin{align*}
    \frac{\partial\varphi}{\partial x} (x,\xi)
    &= \frac{1}{g(x)^2}
    \phi'\pars[\Big]{\frac{h(x,\xi)}{g(x)}}
    \braces[\Big]{\E[h(x,\xi)^2]h(x,\xi) - h(x,\xi)^2 \E[h(x,\xi)]}
    \\
    &= \frac{1}{g(x)^2}
    \phi'\pars[\Big]{\frac{X_x}{g(x)}}
  \braces[\Big]{\E[X_x^2]X_x - X_x^2 \E[X_x]},
  \end{align*}
  with $X_x = h(x,\xi)$.
  Recall that $\phi'$ is increasing ($\phi$ is convex), so, again by the second
  inequality of \eqref{eq:ineqhoverg},
  \[
    \sup_{x\in\intff{\alpha}{1}}
    \abs*{\frac{\partial\varphi}{\partial x} (x,\xi)}
    \leq \frac{2}{\alpha^3 g(1)^2} \phi'\pars[\Big]{\frac{1}{\alpha g(1)}},
  \]
  showing that $f$ is differentiable on $\intof{0}{1}$.
  The very nice trick of \cite{hu_shi_subdiffusive} is to
  consider an independent copy $\tilde{X}_x$ and remark that, by symmetry,
  \begin{align*}
  \E\bracks[\Big]{\frac{\partial\varphi}{\partial x} (x,\xi)}
    &= \frac{1}{2g(x)^2}
    \E\bracks[\Big]{
    \pars[\Big]{
      \phi'\pars[\Big]{\frac{X_x}{g(x)}} -
      \phi'\pars[\Big]{\frac{\tilde{X}_x}{g(x)}}
    }
  \braces[\Big]{\tilde{X}_x^2X_x - X_x^2 \tilde{X}_x}
    }
    \\
    &= \frac{1}{2g(x)^2}
    \E\bracks[\Big]{\tilde{X}_x X_x
    \pars[\Big]{
      \phi'\pars[\Big]{\frac{X_x}{g(x)}} -
      \phi'\pars[\Big]{\frac{\tilde{X}_x}{g(x)}}
    }
    \pars{\tilde{X}_x - X_x}
    } \leq 0,
  \end{align*}
  because the two differences in the expectation have
  opposite signs.

  Finally, we use the hypothesis \eqref{eq:hypfactphi} to prove the continuity of
  $f$ at $0$.
  Let $\varepsilon > 0$ be such that
  $\E\bracks*{\phi((1+\varepsilon)\renorm{\xi})} < \infty$.
  By continuity of $g$, whenever $\alpha$ is small enough,
  $g(0) / g(\alpha) \leq 1 + \varepsilon$.
  Use the first inequality of~\eqref{eq:ineqhoverg} and
  the fact that $\phi$ is non-negative and convex, to obtain that, for any $x \in
  \intff{0}{\alpha}$,
  \[
    \varphi(x,\xi) \leq \max\pars*{\phi(0),
    \phi((1 + \varepsilon)\renorm{\xi})}s
  \]
  and, by dominated convergence, the continuity of $f$ at $0$, so that
  $f(0) \leq f(1)$.
\end{proof}


\begin{thebibliography}{10}

\bibitem{addario-berry_resi}
Louigi Addario-Berry, Nicolas Broutin, and G\'{a}bor Lugosi.
\newblock Effective resistance of random trees.
\newblock {\em Ann. Appl. Probab.}, 19(3):1092--1107, 2009.

\bibitem{aidekon_2007}
Elie A\"{i}d\'{e}kon.
\newblock Transient random walks in random environment on a {G}alton-{W}atson
  tree.
\newblock {\em Probab. Theory Related Fields}, 142(3-4):525--559, 2008.

\bibitem{Elie_speed}
Elie A\"{i}d\'{e}kon.
\newblock Speed of the biased random walk on a {G}alton-{W}atson tree.
\newblock {\em Probab. Theory Related Fields}, 159(3-4):597--617, 2014.

\bibitem{aidekon_raphelis_scaling}
Elie A\"{i}d\'{e}kon and Lo\"{i}c de~Raph\'{e}lis.
\newblock Scaling limit of the recurrent biased random walk on a
  {G}alton-{W}atson tree.
\newblock {\em Probab. Theory Related Fields}, 169(3-4):643--666, 2017.

\bibitem{andreoletti_chen}
Pierre Andreoletti and Xinxin Chen.
\newblock Range and critical generations of a random walk on {G}alton-{W}atson
  trees.
\newblock {\em Ann. Inst. Henri Poincar\'{e} Probab. Stat.}, 54(1):466--513,
  2018.

\bibitem{andreoletti_debs_number}
Pierre Andreoletti and Pierre Debs.
\newblock The number of generations entirely visited for recurrent random walks
  in a random environment.
\newblock {\em J. Theoret. Probab.}, 27(2):518--538, 2014.

\bibitem{chen_hu_lin}
Dayue Chen, Yueyun Hu, and Shen Lin.
\newblock Resistance growth of branching random networks.
\newblock {\em Electron. J. Probab.}, 23:Paper No. 52, 17, 2018.

\bibitem{de_raphelis_scaling_subdiffusive}
Lo\"{i}c {de Raph\'{e}lis}.
\newblock {Scaling limit of the subdiffusive random walk on a Galton--Watson
  tree in random environment}.
\newblock {\em arXiv e-prints}, page arXiv:1608.07061, Aug 2016.

\bibitem{doyle_snell}
Peter~G. Doyle and James~Laurie Snell.
\newblock {\em Random walks and electric networks}, volume~22 of {\em Carus
  Mathematical Monographs}.
\newblock Mathematical Association of America, Washington, DC, 1984.

\bibitem{faraud_2011}
Gabriel Faraud.
\newblock A central limit theorem for random walk in a random environment on
  marked {G}alton-{W}atson trees.
\newblock {\em Electron. J. Probab.}, 16:no. 6, 174--215, 2011.

\bibitem{faraud_hu_shi}
Gabriel Faraud, Yueyun Hu, and Zhan Shi.
\newblock Almost sure convergence for stochastically biased random walks on
  trees.
\newblock {\em Probab. Theory Related Fields}, 154(3-4):621--660, 2012.

\bibitem{hu_local_times_subdiff}
Yueyun Hu.
\newblock Local times of subdiffusive biased walks on trees.
\newblock {\em J. Theoret. Probab.}, 30(2):529--550, 2017.

\bibitem{hu_shi_subdiffusive}
Yueyun Hu and Zhan Shi.
\newblock A subdiffusive behaviour of recurrent random walk in random
  environment on a regular tree.
\newblock {\em Probab. Theory Related Fields}, 138(3-4):521--549, 2007.

\bibitem{kahane_peyriere}
Jean-Pierre Kahane and Jacques Peyri\`{e}re.
\newblock Sur certaines martingales de {B}enoit {M}andelbrot.
\newblock {\em Advances in Math.}, 22(2):131--145, 1976.

\bibitem{shen_lin_harmonic_biased}
Shen Lin.
\newblock {Harmonic measure for biased random walk in a supercritical
  Galton--Watson tree}.
\newblock {\em ArXiv e-prints}, July 2017.

\bibitem{liu_generalized}
Quansheng Liu.
\newblock On generalized multiplicative cascades.
\newblock {\em Stochastic Process. Appl.}, 86(2):263--286, 2000.

\bibitem{Lyons_rwperco}
Russell Lyons.
\newblock Random walks and percolation on trees.
\newblock {\em Ann. Probab.}, 18(3):931--958, 1990.

\bibitem{lyons_1992}
Russell Lyons.
\newblock Random walks, capacity and percolation on trees.
\newblock {\em Ann. Probab.}, 20(4):2043--2088, 1992.

\bibitem{lyons_biggins}
Russell Lyons.
\newblock A simple path to {B}iggins' martingale convergence for branching
  random walk.
\newblock In {\em Classical and modern branching processes ({M}inneapolis,
  {MN}, 1994)}, volume~84 of {\em IMA Vol. Math. Appl.}, pages 217--221.
  Springer, New York, 1997.

\bibitem{lyons_pemantle_rwre}
Russell Lyons and Robin Pemantle.
\newblock Random walk in a random environment and first-passage percolation on
  trees.
\newblock {\em Ann. Probab.}, 20(1):125--136, 1992.

\bibitem{LPP95}
Russell Lyons, Robin Pemantle, and Yuval Peres.
\newblock Ergodic theory on {G}alton-{W}atson trees: speed of random walk and
  dimension of harmonic measure.
\newblock {\em Ergodic Theory Dynam. Systems}, 15(3):593--619, 1995.

\bibitem{LPP_biased}
Russell Lyons, Robin Pemantle, and Yuval Peres.
\newblock Biased random walks on {G}alton-{W}atson trees.
\newblock {\em Probab. Theory Related Fields}, 106(2):249--264, 1996.

\bibitem{LyonsPeres_book}
Russell Lyons and Yuval Peres.
\newblock {\em Probability on trees and networks}, volume~42 of {\em Cambridge
  Series in Statistical and Probabilistic Mathematics}.
\newblock Cambridge University Press, New York, 2016.

\bibitem{neveu_multiplicative}
J.~Neveu.
\newblock Multiplicative martingales for spatial branching processes.
\newblock In {\em Seminar on {S}tochastic {P}rocesses, 1987 ({P}rinceton, {NJ},
  1987)}, volume~15 of {\em Progr. Probab. Statist.}, pages 223--242.
  Birkh\"auser Boston, Boston, MA, 1988.

\bibitem{petrov_limit}
Valentin~V. Petrov.
\newblock {\em Limit theorems of probability theory}, volume~4 of {\em Oxford
  Studies in Probability}.
\newblock The Clarendon Press, Oxford University Press, New York, 1995.
\newblock Sequences of independent random variables, Oxford Science
  Publications.

\bibitem{rou_ddroprwre}
Pierre {Rousselin}.
\newblock {Dimension Drop for Transient Random Walks on Galton--Watson Trees in
  Random Environments}.
\newblock {\em arXiv e-prints}, page arXiv:1711.07920, Nov 2017.

\bibitem{Rou2018}
Pierre Rousselin.
\newblock Invariant measures, hausdorff dimension and dimension drop of some
  harmonic measures on galton--watson trees.
\newblock {\em Electron. J. Probab.}, 23:1--31, 2018.

\bibitem{vonbahr1965}
Bengt von Bahr and Carl-Gustav Esseen.
\newblock Inequalities for the $r$th absolute moment of a sum of random
  variables, $1 \leqq r \leqq 2$.
\newblock {\em Ann. Math. Statist.}, 36(1):299--303, 02 1965.

\end{thebibliography}
\end{document}